\documentclass{article}
\usepackage{geometry}
\usepackage{amsmath,amssymb,amsthm,amsfonts,color,enumerate}
\usepackage{mathrsfs}
\usepackage{bookmark}
\usepackage{bbm}
\usepackage[capitalise]{cleveref}
\usepackage{varwidth}
\usepackage[backend=biber, 
  url=false,                  
  isbn=false,   
  doi=false,    
  maxnames=99
]{biblatex}
\usepackage{url}
\usepackage{ulem}
\usepackage{appendix}
\usepackage{tikz-cd}
\usepackage{quiver}
\usepackage{footmisc}

\newtheorem{Definition}{Definition}[section]

\newtheorem{Theorem}{Theorem}[section]
\newtheorem{Lemma}{Lemma}[section]
\newtheorem{Main Theorem}{Main Theorem}
\newtheorem{Remark}{Remark}[section]

\newtheorem{Corollary}{Corollary}[section]
\newtheorem{Claim}{Claim}[Theorem]

\newtheorem{Question}{Question}
\newtheorem{Test question}{Test question}

\newtheorem*{Open problem}{Open problem}
\newtheorem{Key Lemma}{Key Lemma}

\title{Identity crisis on tightly $C^{(2)}$-compact cardinals}

\author{Fengju Zhang}

\begin{document}
\maketitle

\begin{abstract}
    Bagaria and Goldberg [Question 4.7, \cite{bagairagoldbergreflectingmeasures}] asked if $\kappa$ being tightly $C^{(n)}$-compact implies $\kappa\in C^{(n+1)}$ for natural numbers $n\ge 1$?
    Usuba gave a negative answer for $n=1$.
    In this paper we give a negative answer for $n=2$.
\end{abstract}

\section{Introduction}
Large cardinals are a central area of research in set theory, and many questions about their properties cannot be decided in $\mathsf{ZFC}$. 
A classical example is how large is the least strongly compact cardinal. 
In 1976, Magidor \cite{magidor1976large} proved that the least strongly compact cardinal could be the least supercompact cardinal while it could be the least measurable cardinal. 
Thus $\mathsf{ZFC}$ alone does not determine how large the least strongly compact cardinal is in comparison with measurable cardinals and with supercompact cardinals. 
This is the identity crisis phenomenon studied by Magidor.

This situation changes under additional axioms. Recently Goldberg \cite{Goldbergultrapoweraxiom} introduced a set-theoretic principle called the Ultrapower Axiom ($\mathsf{UA}$), 
and proved that under $\mathsf{UA}$ the class of strongly compact cardinals coincides with the class of supercompact cardinals, except at the measurable limits of supercompact cardinals. 
By Menas's theorem \cite{menas1974strong}, this result is optimal. Consequently, under $\mathsf{UA}$ the least strongly compact cardinal must be the least supercompact cardinal.

A parallel development has taken place for \(C^{(n)}\)-extendible cardinals. Recall that
\[
C^{(n)}=\{\alpha:V_{\alpha}\prec_{\Sigma_{n}}V\}.
\]
In the study of Vopěnka's Principle, Bagaria et. al.  \cite{BagariaMathiasDefinable} introduced \(C^{(n)}\)-extendible cardinals. 
Bagaria and Goldberg \cite{bagairagoldbergreflectingmeasures} later showed that \(C^{(n)}\)-extendibility admits a characterization in terms of the existence of certain ultrafilters which is closely analogous to supercompactness, so it is naturally viewed as a strengthening of supercompactness. They also introduced tightly \(C^{(n)}\)-compact cardinals, a corresponding generalization of strong compactness, and proved the following result:

\begin{Theorem}[Bagaria, Goldberg \cite{bagairagoldbergreflectingmeasures}]
Under $\mathsf{UA}$, 
the class of tightly \(C^{(n+1)}\)-compact cardinals coincides with the class of \(C^{(n)}\)-extendible cardinals, except at the measurable limits of $C^{(n)}$-extendible cardinals.
\end{Theorem}

Since every \(C^{(n)}\)-extendible cardinal is \(\Sigma_{n+2}\)-correct, 
it follows under $\mathsf{UA}$ that every tightly \(C^{(n+1)}\)-compact cardinal is \(\Sigma_{n+2}\)-correct. 
Bagaria and Goldberg then asked whether this amount of correctness follows outright from tight compactness alone.

\begin{Question}[Bagaria, Goldberg \cite{bagairagoldbergreflectingmeasures}]
Does \(\kappa\) being tightly \(C^{(n)}\)-compact imply \(\kappa\in C^{(n+1)}\) for natural numbers $n\ge 1$?
\end{Question}

For \(n=1\), Usuba gave a negative answer:

\begin{Theorem}[Usuba \cite{Usubaprivate}]
\label{Theorem: Usuba}
Modulo a supercompact cardinal, it is consistent that the  least tightly \(C^{(1)}\)-compact cardinal is not $\Sigma_2$-correct.
\end{Theorem}

Usuba's original proof is similar to that of Theorem 3.5 in \cite{magidor1976large}.
More precisely, starting from a model of set theory containing a supercompact cardinal \(\kappa\)
and no measurable cardinals above \(\kappa\),
Usuba used a Magidor iteration $\mathbb P$ of Prikry forcing to destroy all measurable cardinals below \(\kappa\).
Consequently, in \(V^{\mathbb P}\), \(\kappa\) is the unique measurable cardinal.
He then showed that the tight \(C^{(1)}\)-compactness of \(\kappa\) is preserved by \(\mathbb P\).
Since the sentence \(\phi\equiv\) ``there exists a measurable cardinal'' is a \(\Sigma_2\)-sentence, 
and since \((V^{\mathbb P})_\kappa \not\models \phi\),
it follows that \(\kappa\) is not \(\Sigma_2\)-correct in \(V^{\mathbb P}\).

The main result of this paper is the corresponding negative answer for \(n=2\):

\begin{Theorem}
\label{Theorem: tightly C(2) may not}
It is consistent, modulo an extendible cardinal and a supercompact cardinal above it, that the least tightly \(C^{(2)}\)-compact cardinal is not $\Sigma
_3$-correct.
\end{Theorem}

The paper is organized as follows. Section 2 gives an alternative proof of Usuba's negative result for \(n=1\).
Instead of using a Magidor iteration of Prikry forcing,
we use an Easton support iterated forcing to add non-reflecting stationary subsets to all the measurable cardinals below $\kappa$. 
Section 3 proves the compatibility of the Kimchi--Magidor theorem (i.e., the class of strongly compact cardinals coincides with the class of supercompact cardinals except at the measurable limits of supercompact cardinals) with an extendible cardinal. 
Section 4 proves Theorem 1.3, and Section 5 discusses the case \(n\ge 3\).

\medskip
\noindent\textbf{Notation.} In addition to the notation introduced above, \(\mathcal{I}\) denotes the class of inaccessible cardinals, and \(\mathcal{S}\) denotes the class of supercompact cardinals. 
For regular cardinals \(\alpha\) and \(\beta\), \(\mathbb{NR}(\alpha,\beta)\) denotes the forcing that adds a non-reflecting stationary subset of $\alpha$,
all of whose elements have cofinality \(\beta\).

\section{Proof for $n=1$}
In this section,
we will prove that tightly $C^{(1)}$-compact cardinal may not be $\Sigma_2$-correct.
\begin{Definition}
    If $j:V\to M$ is an elementary embedding, for any $A\in V$ define
    $t_j(A)=\min \{|B|^M:j''A\subseteq B\}$.
\end{Definition}

\begin{Definition}[Bagaria, Goldberg \cite{bagairagoldbergreflectingmeasures}]
    Let $n$ be a natural number with $n\ge 1$.
    If $\lambda\in C^{(n)}$, we say $\kappa$ is tightly $\lambda$-$C^{(n)}$-compact if there exists $j:V\to M$ such that $t_j(\lambda)\in (C^{(n)})^M\cap j(\kappa)$.
    We say $\kappa$ is tightly $C^{(n)}$-compact if for every $\lambda\in C^{(n)}$, $\kappa$ is tightly $\lambda$-$C^{(n)}$-compact.
\end{Definition}

\begin{Lemma}[Bagaria, Goldberg \cite{bagairagoldbergreflectingmeasures}]
\label{Lemma: closed}
    If $i:V\to M$ and $k:M\to N$ are elementary embeddings, and $k$ is close to $M$, i.e., for every $A\in N$ we have $k^{-1}{''}A\in M$, then setting $j=k\circ i$, we have $t_j=t_k\circ t_i$.
\end{Lemma}

\begin{Lemma}[Bagaria, Goldberg\cite{bagairagoldbergreflectingmeasures}]
\label{Lemma: 2}
    If $j:V\to M$ is an ultrapower embedding induced by some $\kappa$-complete ultrafilter, $\lambda$ is a cardinal, and $\mathrm{cf}(\lambda)\le \kappa$, then $t_j(\lambda)=\sup_{\delta<\lambda}t_j(\delta)$.
\end{Lemma}

\begin{Lemma}[Bagaria, Goldberg\cite{bagairagoldbergreflectingmeasures}]
\label{Lemma: 3}
    If $\kappa<\lambda$ are cardinals, $\mathrm{cf}(\lambda)<\kappa$, and $j_U:V\to M_U$ is the ultrapower embedding induced by a $\kappa$-complete ultrafilter $U$ on some set of size less than $\lambda$, then $t_j(\lambda)=j(\lambda)$.
\end{Lemma}

\begin{proof}[Proof of \autoref{Theorem: Usuba}]
    We present a proof different from Usuba's. The proof below is essentially the same as that in Theorem 2 in  \cite{Aptercummingsidentity}.
    Assume that \(\kappa\) is a supercompact cardinal and that \(\mathrm{GCH}\) holds.
    We further assume that there is no measurable cardinal above \(\kappa\).
    Define an Easton-support iteration \(P_\kappa=\langle P_\alpha,\dot Q_\alpha\rangle_{\alpha<\kappa}\) as follows:
    \(P_0=\mathsf{Add}(\omega,1)\);
    suppose \(P_\alpha\) has been defined and \(\alpha\) is a measurable cardinal,
    then let \(\dot Q_\alpha=\mathbb{NR}(\alpha,\omega)\);
    otherwise, let \(\dot Q_\alpha\) be a trivial forcing.
    Assume \(G_\kappa\) is \(P_\kappa\)-generic over \(V\).
    By the gap forcing theorem \cite{hamkins2001gap},
    there is no measurable cardinal below \(\kappa\).
    Now we verify that \(\kappa\) is tightly \(C^{(1)}\)-compact in \(V[G_\kappa]\):
    pick \(\lambda\in C^{(1)}-(\kappa+1)\),
    and let \(j:V\to M\) be the ultrapower embedding induced by some \(\kappa\)-complete fine normal measure on \(\mathcal P_\kappa(\lambda^+)\).
    Since \(\kappa\) is measurable in \(M\),
    pick a normal measure \(U\) of Mitchell order \(0\) on \(\kappa\) in \(M\),
    and let \(i:M\to N\) be the ultrapower induced by \(U\).
    Let \(k=i\circ j:V\to N\) be the composition.

    \begin{Claim}
        \(k\) is a tightly \(\lambda\)-\(C^{(1)}\)-compact embedding.
    \end{Claim}
\label{Claim: 1}
    \begin{proof}[Proof of Claim]
        Notice that \(t_j(\lambda)=\lambda\in (C^{(1)})^M\). Since \(i\) is close to \(M\), by \autoref{Lemma: closed} we have
        \[
        t_k(\lambda)=(t_i\circ t_j)(\lambda)=t_i(t_j(\lambda))=t_i(\lambda).
        \]
        We show that \(t_i(\lambda)\in (C^{(1)})^N\):
        if \(\operatorname{cf}(\lambda)<\kappa\), then by \autoref{Lemma: 3}, \(t_i(\lambda)=i(\lambda)\in (C^{(1)})^N\).
        If \(\operatorname{cf}(\lambda)=\kappa\), then \(\lambda\in \lim(C^{(1)})\), so \(M\models\lambda\in \lim(C^{(1)})\).
        Let \(\delta\in \lambda\cap C^{(1)}\) be a successor element of \(C^{(1)}\) above $\kappa$.
 Then \(\operatorname{cf}^M(\delta)=\operatorname{cf}(\delta)=\omega\) and \(\delta\in (C^{(1)})^M\), 
 so by \autoref{Lemma: 3}, \(t_i(\delta)=i(\delta)\in (C^{(1)})^N\).
        By \autoref{Lemma: 2},
        \[t_i(\lambda)=\sup_{\delta<\lambda}t_i(\delta)=\sup_{\delta<\lambda}i(\delta)\in (\lim(C^{(1)}))^N,
        \]
        hence \(t_i(\lambda)\in (C^{(1)})^N\).
        If \(\operatorname{cf}(\lambda)>\kappa\), then by \(\mathrm{GCH}\), \(i(\lambda)=\lambda\), so \(t_i(\lambda)=\lambda\in (C^{(1)})^N\).
        Thus \(k\) is a tightly \(\lambda\)-\(C^{(1)}\)-compact embedding.
        \renewcommand{\qedsymbol}{$\dashv$}
    \end{proof}

    Let \(j(P_\kappa)=P_\kappa*\dot Q_{\kappa}*P^M_{\kappa,j(\kappa)}\).
    Since there is no measurable cardinal above \(\kappa\),
    \(P^M_{\kappa,j(\kappa)}\) is trivial on the interval \((\kappa,\lambda^+]\).
    Then
    \[
    k(P_\kappa)=i(P_\kappa*\dot Q_\kappa*P^M_{\kappa,j(\kappa)})
    =P_{\kappa}*P^N_{\kappa,i(\kappa)}*i(\dot Q_\kappa)*i(P^M_{\kappa,j(\kappa)}).
    \]
    By \(\mathrm{GCH}\) in $V[G_\kappa]$,
 there are only \(2^\kappa=\kappa^+\) many dense subsets of \(P_{\kappa,i(\kappa)}^N*i(\dot Q_\kappa)\) in \(N[G_\kappa]\).
    Since \(N[G_\kappa]^\kappa\cap V[G_\kappa]\subseteq N[G_\kappa]\) and \(P_{\kappa,i(\kappa)}^N*i(\dot Q_\kappa)\) is \(\kappa^+\)-closed,
    so we can find \(H\in V[G_\kappa]\) that is \(P_{\kappa,i(\kappa)}^N*i(\dot Q_\kappa)\)-generic over \(N[G_\kappa]\).

    To perform the lifting,
    we still need an \(i(P^M_{\kappa,j(\kappa)})\)-generic filter over \(N[G_\kappa*H]\).
    To do this,
    define the term forcing
    \[
    \mathcal T=\{\sigma\in M:\mathbbm{1}\Vdash^M_{P_\kappa*\dot Q_\kappa} \sigma\in P^M_{\kappa,j(\kappa)}\}
    \]
    and define \(\sigma\le_\mathcal T \tau\) if and only if
    \[
    \mathbbm{1}\Vdash^M_{P_\kappa*\dot Q_\kappa}\sigma \le_{P^M_{\kappa,j(\kappa)}} \tau.
    \]
    Even though \(\mathcal T\) is a class, we can use Scott's trick to make it a set.
    One may also assume that $\mathcal T$ is of minimum size in $M$.
    It is not hard to see that \(\mathcal T\in M\) is a \(\lambda^{++}\)-strategically closed forcing in both $M$ and $V$ and that \(M\models |\mathcal T|\le j(\kappa)\).
    Therefore in $V$,
 there are only \((2^{j(\kappa)})^M\le \lambda^{++}\) many dense subsets of \(\mathcal T\) in \(M\).
    Since \(M^{\lambda^+}\subseteq M\),
    we can find a \(\mathcal T\)-generic filter \(K\) over \(M\) in $V$.
    Since all dense subsets of \(i(\mathcal T)\) in \(N\) are of the form \(i(f)(\kappa)\) for some function \(f:\kappa\to \mathcal P (\mathcal T)\) in $M$, 
    and \(\mathcal T\) is \(\kappa^+\)-strategically closed,
    \(i''K\) generates an \(i(\mathcal T)\)-generic filter \(J\) over \(N\).
    Define \(I=\{\sigma_{G_\kappa*H}:\sigma \in J\}\).
    It is not hard to see that \(I\) is \(i(P^M_{\kappa,j(\kappa)})\)-generic over \(N[G_\kappa*H]\).
    By the standard lifting argument  \cite{CummingsiteratedforcingHandbook},
    we can lift \(k\) to \(k^*:V[G_\kappa]\to N[G_\kappa][H][I]\) in $V[G_\kappa]$.

    Finally, it remains to show that
    \[
    t_{k^*}(\lambda)=t_k(\lambda)\in (C^{(1)})^{N[G_\kappa][H][I]}.
    \]
    For the first equality,
    suppose otherwise that \(t_{k^*}(\lambda)<t_k(\lambda)\).
    Let \(X\in N[G_\kappa*H*I]\) be such that \(X\) covers \(k^{*\prime\prime}\lambda=\) and \(N[G_\kappa*H*I]\models |X|=t_{k^*}(\lambda)\).
    Since \(i(P^M_{\kappa,j(\kappa)})\) is \((i(\lambda)^+)^{N[G_\kappa*H]}\)-strategically closed and $t_k(\lambda)=t_i(\lambda)\le i(\lambda)$,
    \(X\in N[G_\kappa*H]\).
    Now since
    \[
    |P_\kappa*P^N_{\kappa,i(\kappa)}*i(\dot Q_\kappa)|=i(\kappa)<\lambda\le t_{k^*}(\lambda),
    \]
    by the \(\lambda\)-c.c. of \(P_\kappa*P^N_{\kappa,i(\kappa)}*i(\dot Q_\kappa)\),
    there exists \(Y\in N\) such that \(Y\supseteq X\supseteq k^*{''}\lambda=k''\lambda\) and \(|Y|^N=t_{k^*}(\lambda)<t_k(\lambda)\),
    which is a contradiction.
    For the second membership,
    because \(P_\kappa*P^N_{\kappa,i(\kappa)}*i(\dot Q_{\kappa})\) is small relative to \(t_k(\lambda)\),
    we have \(t_k(\lambda)\in (C^{(1)})^{N[G_\kappa*H]}\).
    Since \(i(P^M_{\kappa,j(\kappa)})\) is \((i(\lambda)^+)^N\)-strategically closed and $t_k(\lambda)=t_i(\lambda)\le i(\lambda)$,
    it follows that $t_k(\lambda)\in (C^{(1)})^{N[G_\kappa][H][I]}$.
\end{proof}

\begin{Corollary}[Usuba]
    Modulo a supercompact cardinal, it is consistent that the least tightly $C^{(1)}$-compact cardinal is not $\Sigma_2$-correct.
\end{Corollary}

\begin{proof}[Proof]
   
 Consider the model $V[G_\kappa]$ in \autoref{Theorem: Usuba}.
 Define the \( \Sigma_2 \)-sentence \( \phi \equiv \) ``there exists a measurable cardinal.'' 
 \( \phi \) is true in \( V[G_\kappa] \). However, because \( \kappa \) is the least measurable cardinal in $V[G_\kappa]$,
 there are no measurable cardinals below it; consequently, \( \phi \) fails to hold in the initial segment \( V[G_\kappa]_\kappa \).
This immediately rules out \( \Sigma_2 \)-correctness for \( \kappa \) in $V[G_\kappa]$. 
\end{proof}

\section{Kimchi-Magidor Theorem is compatible with an extendible cardinal}
In this section we will show that Kimchi-Magidor Theorem is compatible with an extendible cardinal.
\begin{Definition}
    We define the global Menas function $\ell$ as follows: $\langle\alpha,\beta+1\rangle\in \ell$ iff
    \begin{enumerate}
        \item $\ell''\alpha\subseteq\alpha$;
        \item $\alpha\in \mathcal I-\mathcal S$;
        \item $\beta$ is the least ordinal such that “there is no $\mathcal P_\alpha(\beta)$-fine normal measure $U$ with $j_U(\ell\upharpoonright\alpha)(\alpha)>\beta$”.
    \end{enumerate}
\end{Definition}

Notice that for any $\delta\in C^{(2)}$, $\ell\cap V_\delta=\ell^{V_\delta}$ and $\ell''\delta\subseteq\delta$.
It is easy to see that $\ell$ is undefined on $\mathcal S$, and for every $\xi\in \mathcal S$, $\ell\upharpoonright\xi:\xi\to V_\xi$ is a function on $\xi$ with the \textit{Menas property},
which means for all $\alpha\ge \xi$,
there exists some fine normal measure $U$ on $\mathcal P_\xi(\alpha)$ such that $j_U(\ell \upharpoonright\xi)(\xi)>\alpha$.

\begin{Definition}
    Let $\mathcal C$ denote the smallest closed unbounded class containing $\mathcal S\cup\{\omega\}$.
    Write $\Phi(\alpha,\beta,Q)$ for “$Q\in V_\beta$ is an $\alpha$-directed closed forcing”.
    Let $\ell$ be the global Menas function defined above. Now define the Easton support iteration $\mathbb P^{\ell}=\langle \mathbb P^{\ell}_\beta,\dot{\mathbb Q}^{\ell}_\beta\rangle_{\beta<\mathsf{Ord}}$:
    Let $\mathbb P^{\ell}_0=\mathsf{Add}(\omega,1)$;
    assume $\mathbb P^{\ell}_\beta$ is defined and $\beta\in \mathrm{dom}(\ell)$,
    since $\mathcal C$ is closed,
    there exists $\gamma\in \mathcal C$ such that $\gamma=\sup (\beta\cap \mathcal C)$.
    Now define $\dot{\mathbb Q}^{\ell}_\beta=\dot{\mathbb Q}^{\ell}_{\beta,0}*\dot{\mathbb Q}^{\ell}_{\beta,1}$, where
    $$\dot{\mathbb Q}^{\ell}_{\beta,0}=\bigoplus\{Q:V^{\mathbb P^{\ell}_\beta}\models\Phi(\beta,\ell(\beta),Q)\}$$
    and $\dot{\mathbb Q}^\ell_{\beta,1}=\mathbb{NR}(\ell(\beta)^+,\gamma^*)$,
    where $\gamma^*$ is the least regular cardinal $\ge\gamma$.
    Otherwise let $\dot{\mathbb Q}^{\ell}_\beta$ be trivial.
    We call $\mathbb P^\ell$ \textit{the standard forcing for KMT} (Kimchi-Magidor Theorem).
\end{Definition}

\begin{Theorem}[Kimchi–Magidor\cite{kimchi1980independence}]
\label{Theorem :KMT}
    Let $\mathbb P^\ell$ be the standard forcing for KMT, and suppose $G$ is $(V,\mathbb P^\ell)$-generic. Then the following hold:
    \begin{enumerate}
        \item $\mathcal S=\mathcal S^{V[G]}$;
        \item every $\delta\in \mathcal S$ is Laver indestructible in $V[G]$;
        \item $V[G]\models \mathrm{KMT}$, i.e.,
        the class of strongly compact cardinals coincides with the class of supercompact cardinals, except at measurable limits.
    \end{enumerate}
\end{Theorem}

\begin{Theorem}[Bagaria, Goldberg\cite{bagairagoldbergreflectingmeasures}]
\label{Theorem: extendible}
    A cardinal $\kappa$ is extendible iff for all $\lambda\in C^{(2)}-\kappa$,
    there exists a fine normal measure $U$ such that $T^2_{\kappa,\lambda}\in U$.
\end{Theorem}

\begin{Theorem}
\label{Theorem: KMT exntedible}
    It is consistent, modulo an extendible cardinal and a supercompact cardinal above it, that KMT is compatible with an extendible cardinal.
\end{Theorem}

\begin{proof}[Proof]
    Let $\kappa$ be an extendible cardinal and $\theta>\kappa$ be supercompact.
    Let $\mathbb P^\ell$ be the standard forcing for KMT as defined above, and let $G$ be $(V,\mathbb{P}^\ell)$-generic.
    Then by \autoref{Theorem :KMT}, $\theta$ remains supercompact in $V[G]$.
    We show below that $\kappa$ is $\theta$-extendible in $V[G]$, then $V[G]_\theta$ is a model where “there exists an extendible cardinal” and KMT holds.

    Now pick any $\eta\in C^{(2)}-(\theta+1)$. Since $\mathbb P_{\eta,\infty}$ is $\eta$-strategically closed, it suffices to show that $\kappa$ is $\theta$-extendible in $V[G_\eta]$.
    By \autoref{Theorem: extendible},
    Let $U$ be a $2$-reflecting fine normal measure on $\mathcal P_\kappa(\eta)$, and let $j_U:V\to M_U$ be the ultrapower induced by $U$.
    Since $M_U\models \eta\in C^{(2)}$ and $V_\eta\models$ "$\theta$ is supercompact", then $\theta$ is supercompact in $M_U$.
    Because $\eta\in (C^{(2)})^{M_U}$, we have $j_U(\ell)\upharpoonright \eta=\ell\upharpoonright \eta$, and hence
    $j_U(\mathbb P^\ell_\kappa*\mathbb P^{\ell}_{\kappa,\eta})=\mathbb P^\ell_\eta* P_{tail}*j_U(\mathbb P^\ell_{\kappa,\eta})$,
    where $P_{tail}=j_U(\mathbb P^\ell_\kappa)_{\eta,j_U(\kappa)}$.
    Since $\mathbb P^\ell_{\kappa,\eta}$ is $\kappa$-directed closed and $M_U^\eta\subseteq M_U$, $j_U(\mathbb P^\ell_{\kappa,\eta})$ is $\eta^+$-directed closed.
    As $|\mathbb P^\ell_\eta|\le \eta$, there exists a master condition $q^*\in j_U(\mathbb P^\ell_{\kappa,\eta})$.
    Let $H$ be $(V[G_\eta],P_{tail}*j_U(\mathbb P^\ell_{\kappa,\eta}))$-generic with $q^*\in H$, and lift $j_U$ to $j_U^*:V[G_\eta]\to M_U[G_\eta][H]$.
    Let $W$ be the fine normal measure on $\mathcal P_{\kappa}^{V[G_\eta]}(\theta)$ induced by $j_U^*$.
    By the definition of $\mathbb P^\ell$, $\dot{\mathbb Q}^\ell_\theta$ is trivial and $\mathbb P^\ell_{\theta,\eta}$ is $|2^{{\theta}^{<\kappa}}|^+$-strategically closed, so $W\in V[G_\theta]$.
    Moreover, $j_U(\mathbb P^\ell_\eta)$ preserves the supercompactness of $\theta$, so $\theta$ is supercompact in $M_U[G_\eta][H]$, and consequently $\theta\in (C^{(2)})^{M_U[G_\eta][H]}$.
    This shows that $W$ is a $2$-reflecting fine normal measure on $\mathcal P^{V[G_\theta]}_\kappa(\theta)=\mathcal P^{V[G_\eta]}_\kappa(\theta)$.
    Hence $\kappa$ is extendible in $V[G_\theta]_\theta=V[G]_\theta$, completing the proof.
\end{proof}

\section{Proof of $n=2$}
This section shows the main result of the paper:
the least tightly \(C^{(2)}\)-compact cardinal may not be \(\Sigma_3\)-correct.

\begin{Definition}[Club shooting]
    If $\xi$ is a Mahlo cardinal, let $S_\xi\subseteq\xi$ be the set of inaccessible cardinals below $\xi$.
    Define $\mathbb C_\xi$ as the set of bounded closed subsets of $\xi$ that are disjoint from $S_\xi$, ordered by end-extension.
    Clearly $\mathbb C_\xi$ shoots a club disjoint from $S_\xi$.
\end{Definition}

\begin{Lemma}
\label{Lemma: club shooting}
    $\mathbb C_\xi$ is $\xi^+.c.c$, and for every $\chi<\xi$, $\mathbb C_\xi$ contains a dense set $\mathbb D_{\chi,\xi}$ such that $\mathbb D_{\chi,\xi}$ is $\chi$-directed closed.
\end{Lemma}

\begin{proof}[Proof]
    Since $|\mathbb C_\xi|=\xi$, $\mathbb C_\xi$ is $\xi^+.c.c$.
    Define $\mathbb D_{\chi,\xi}=\{p\in \mathbb C_{\chi,\xi}:\max p\ge \chi\}$.
    Clearly $\mathbb D_{\chi,\xi}$ is dense in $\mathbb C_\xi$ and is $\chi$-directed closed.
\end{proof}

\begin{Definition}[Bagaria, Goldberg\cite{bagairagoldbergreflectingmeasures}]
    Let $\kappa\le \lambda$. We say $U$ is an $n$-reflecting measure on $\mathcal P_\kappa(\lambda)$ iff $U$ is $\kappa$-complete and $T^n_{\kappa,\lambda}\in U$, where $T^n_{\kappa,\lambda}=\{\sigma\in \mathcal P_\kappa(\lambda):\mathrm{ot}(\sigma)\in C^{(n)}\}$.
\end{Definition}

\begin{proof}[Proof of \autoref{Theorem: tightly C(2) may not}]
    By \autoref{Theorem: KMT exntedible}, we may assume that the ground model $V$ satisfies KMT and that $\kappa$ is extendible.
    Further assume that there is no supercompact cardinal above $\kappa$; then by KMT there is no strongly compact cardinal above $\kappa$.
    Let $\mathbb T$ be the canonical forcing for $\mathrm{GCH}$; $\mathbb T$ is a gap forcing.
    By [Theorem 5.2, \cite{tsaprounis2018c}], $\mathbb T$ preserves all $C^{(n)}$-extendible cardinals.
    Standard arguments show that $\mathbb T$ preserves all supercompact and measurable cardinals, and hence all strongly compact cardinals.
    Since $\mathbb T$ is a gap forcing, it does not force any new large cardinals, so $\mathbb T$ preserves KMT.
    We may still denote by $V$ the model obtained by forcing with $\mathbb T$, then $V\models\mathrm{GCH}\wedge \mathrm{KMT}$.

    By [Lemma 3.5, \cite{GoldbergoptimalityHODdichotomy}] and [Theorem 3.8, \cite{GoldbergUsuba}], there exists a class forcing $\mathbb Q$ with the following properties:
    \begin{enumerate}
        \item $\mathbb Q$ preserves extendible cardinals;
        \item For every $\lambda\in C^{(2)}$, $\mathbb Q=\mathbb Q_\lambda*\mathbb Q_{\lambda,\infty}$ where $\mathbb Q_{\lambda,\infty}$ is $\lambda^+$-directed closed;
        \item $\lambda$ is $\Sigma_2$-correct in $V$ iff $\lambda$ is $\Sigma_2$-indestructibly correct under $\lambda^+$-directed closed posets that preserve $\mathrm{GCH}$ in $V^\mathbb Q$;
        \item If $V\models \mathrm{GCH}$, then $V^\mathbb Q\models \mathrm{GCH}$.
    \end{enumerate}
    Note that our condition (2) differs slightly from that in \cite{GoldbergUsuba} and \cite{GoldbergoptimalityHODdichotomy}: we require $\mathbb Q_{\lambda,\infty}$ to be $\lambda^+$-directed closed rather than $\lambda$-directed, a minor adjustment for technical issue.
    We may assume that $\mathbb Q$ admits a gap at $\omega$.

    Let $V_1=V^\mathbb Q$.
    Now let $\mathcal K$ be the set of $\xi<\kappa$ such that $\xi$ is strongly compact in $V$ and Mahlo in $V_1$.
    Define the Easton support product
    $$\mathbb P=\prod^{\mathrm{Easton}}_{\xi\in \mathcal K}\mathbb C_\xi$$
    and let $V_2=V_1^\mathbb P$.
    Let $H$ be $\mathbb Q$-generic over $V$ and $G$ be $\mathbb P$-generic over $V_1=V[H]$.

\begin{Claim}
    $\mathbb P$ preserves $\mathrm{GCH}$ and cardinals.
\end{Claim}
\begin{proof}[Proof of Claim]
    For any $\alpha\in \mathsf{Card}$, write $\mathbb P=\mathbb P^{\le \alpha}\times \mathbb P^{>\alpha}$.
    Since $\mathbb P^{>\alpha}$ has a dense subset that is $\alpha^+$-closed, it suffices to show that $V_1^{\mathbb P^{\le \alpha}}\models 2^\alpha=\alpha^+$.
    We consider two cases:
    If $\alpha$ is not a limit point of $\mathcal K$, then $V_1\models\mathrm{GCH}$ and $|\mathbb P^{\le \alpha}|\le \alpha$.
    If $\alpha$ is a limit of $\mathcal K$, then $\alpha$ is a strong limit cardinal in $V_1$.
    Pick any $\xi\in \mathcal K\cap \alpha$. If $\xi$ is a successor element of $\mathcal K$, then $\mathbb P^{\le \alpha}=\mathbb P^{<\xi}\times \mathbb P^{[\xi,\alpha]}$, where $|\mathbb P^{<\xi}|<\xi$ (as $\xi$ is Mahlo in $V_1$) and $\mathbb P^{[\xi,\alpha]}$ has a dense subset that is $|\mathbb P^{<\xi}|^+$-directed closed.
    By the above analysis, forcing with $\mathbb P^{\le \alpha}$ preserves “$\alpha$ is a strong limit cardinal”, so $V_1^{\mathbb P^{\le\alpha}}\models 2^\alpha=\alpha^{\mathrm{cf}(\alpha)}$.
    If $\alpha$ is inaccessible, then because $\mathbb P$ is Easton support, $|\mathbb P^{\le\alpha}|=\alpha$, so $V_1^{\mathbb P^{\le\alpha}}\models2^\alpha=\alpha^+$.
    If $\alpha$ is singular, then $\mathbb P^{>\mathrm{cf}(\alpha)}$ does not force any sequence of length $\mathrm{cf}(\alpha)$, so $V_1^{\mathbb P^{\le\alpha}}\models 2^\alpha=\alpha^{\mathrm{cf}(\alpha)}=\alpha^+$.
    \renewcommand{\qedsymbol}{$\dashv$}
\end{proof}

    By \autoref{Lemma: club shooting}, $\mathbb P$ contains an $\omega_1$-directed closed dense subset, so we may regard $\mathbb Q*\mathbb P$ as admitting a gap at $\omega$.
    We now show that there is no strongly compact cardinal below $\kappa$ in $V_2$:
    Otherwise, let $\theta<\kappa$ be strongly compact in $V_2$. If $\theta$ is not Mahlo in $V_1$, then clearly it is not Mahlo in $V_2$.
    If $\theta$ is Mahlo in $V_1$, then because $\mathbb Q*\mathbb P$ admits a gap at $\omega$, $\theta$ is strongly compact in $V$. But then $\mathbb P$ forces that $\theta$ is not Mahlo in $V_2$, contradicting the assumption.
    Hence in $V_2$ there is no strongly compact cardinal below $\kappa$.

    Finally, we prove that $\kappa$ remains tightly $C^{(2)}$-compact in $V_2$.
    First recall that $(C^{(2)})^{V}=(C^{(2)})^{V_1}$, so we simply denote by $C^{(2)}$ the $\Sigma_2$-correct cardinals in $V$ and $V_1$.
    Fix any $\lambda\in C^{(2)}-\kappa$. Choose $\mu\in C^{(2)}-(\lambda+1)$ such that $\kappa<\mathrm{cf}(\mu)<\mu$.
    By \autoref{Theorem: extendible},
    Let $D$ be a $2$-reflecting fine normal measure on $\mathcal P^{V_1}_\kappa(\mu)$, and let $j_1:V_1\to M_1$ be the ultrapower induced by $D$.
    Then $M_1\models\kappa, \lambda,\mu\in C^{(2)}$.
    Because $\mathbb Q*\mathbb P$ admits a gap at $\omega$, the restriction $j=j_1\upharpoonright V:V\to M$ is definable in $V$ and $j_1$ is a lift of $j$.
    Since $(C^{(2)})^{V}=(C^{(2)})^{V_1}$, we have $(C^{(2)})^{M}=(C^{(2)})^{M_1}$, and therefore $M\models \kappa,\mu,\lambda\in C^{(2)}$.
    Because there is no strongly compact cardinal above $\kappa$ in $V$, we have $V_\mu\models$“there is no strongly compact cardinal in $(\kappa,\infty)$”, and consequently $V_\mu^{V_1}\models$“there is no $V$-strongly compact cardinal in $(\kappa,\infty)$”.
    Moreover, $V_\mu^{V_1}=V_\mu^{M_1}\preceq_2 M_1$ and $V_\mu=V_\mu^M$, so $M_1\models$“there is no $M$-strongly compact cardinal in $(\kappa,\mu)$”.

    Since $M^\mu\subseteq M$, there exists a normal measure $U\in M$ on $\kappa$ of Mitchell order $0$ (i.e., $U$ is a normal measure and $\kappa$ is not measurable in $\mathrm{Ult}(M,U)$).
    Let $i:M\to \mathrm{Ult}(M,U)=N$ be the ultrapower induced by $U$.
    Because $M\models \kappa\in C^{(2)}$ and $\mathrm{cf}^M(\kappa)>\omega$, we have $M\models \kappa\in \lim(C^{(2)})$.
    As $\mathrm{crt}(i)=\kappa$, it follows that $N\models \kappa\in \lim (C^{(2)})$, hence $N\models \kappa\in C^{(2)}$.
    Write $j(\mathbb Q)=\mathbb Q_\kappa*j(\mathbb Q)_{\kappa,\infty}$. Since $M\models\mathrm{GCH}$ and $\mathbb Q_{\kappa}\subseteq M_\kappa=V_\kappa$, in $N$ there are only $|i(2^\kappa)|=|i(\kappa^+)|=\kappa^+$ many dense subsets of $i(\mathbb Q_\kappa)$.
    By the definition of $\mathbb Q_\kappa$ and $\kappa\in (C^{(2)})^N$, $i(\mathbb Q_\kappa)_{\kappa,i(\kappa)}$ is $\kappa^+$-directed closed, so a standard argument allows us to lift $i$ to $M[H_\kappa]$.
    Since $M\models\kappa\in C^{(2)}$, the definition of $j(\mathbb Q)$ implies that $j(\mathbb Q)_{\kappa,\infty}$ is $\kappa^+$-directed closed, hence $i$ can be lifted further to $i_1:M_1\to N[i_1(j(H))]$.
    Let $N_1=N[i_1(j(H))]$. It is easy to see that $N_1\models$“$\kappa$ is not $N$-measurable, and there is no $N$-strongly compact cardinal in $(i_1(\kappa),i_1(\mu))$”.
    Moreover, in $N_1$, every $\Sigma_2$-correct cardinal $\eta$ is indestructible under $\eta^+$-directed closed posets that preserve $\mathrm{GCH}$.

    Now set $k=i_1\circ j_1:V_1\to N_1$. Since $i_1$ is close to $M_1$, by \autoref{Lemma: closed} we have
    $$t_k(\lambda)=(t_{i_1}\circ t_{j_1})(\lambda)=t_{i_1}(t_{j_1}(\lambda))=t_{i_1}(\lambda).$$
    We show that $t_{i_1}(\lambda)\in (C^{(2)})^{N_1}$:
    If $\mathrm{cf}(\lambda)<\kappa$, then by \autoref{Lemma: 3}, $t_{i_1}(\lambda)=i_1(\lambda)\in (C^{(2)})^{N_1}$.
    If $\mathrm{cf}(\lambda)=\kappa$, then $\lambda\in \lim(C^{(2)})$, so $M_1\models\lambda\in \lim(C^{(2)})$.
    Let $\delta\in \lambda\cap C^{(2)}$ be a successor element of $C^{(2)}$. Then $\mathrm{cf}^{M_1}(\delta)=\mathrm{cf}(\delta)=\omega$ and $\delta\in (C^{(2)})^{M_1}$, so $t_{i_1}(\delta)=i_1(\delta)\in (C^{(2)})^{N_1}$.
    By \autoref{Lemma: 2},
    $$t_{i_1}(\lambda)=\sup_{\delta<\lambda}t_{i_1}(\delta)=\sup _{\delta<\lambda}i_1(\delta)\in (\lim(C^{(2)}))^{N_1},$$
    hence $t_{i_1}(\lambda)\in (C^{(2)})^{N_1}$.
    If $\mathrm{cf}(\lambda)>\kappa$, then by $\mathrm{GCH}$, $i_1(\lambda)=\lambda$, so $t_{i_1}(\lambda)=\lambda\in (C^{(2)})^{N_1}$.
    Thus $k$ is a tightly $\lambda$-$C^{(2)}$-compact embedding.

    Because $N_1\models$“there is no $N$-strongly compact cardinal in $(i_1(\kappa),i_1(\mu))$”, we can write $k(\mathbb P)=\mathbb P\times \mathbb R\times \mathbb S$ where $\mathbb R=k(\mathbb P)\upharpoonright[\kappa,i_1(\kappa)]$ and $\mathbb S=k(\mathbb P)\upharpoonright[i_1(\mu),k(\kappa))$.
    Since $N_1\models t_k(\lambda)^+\le i_1(\lambda)^+<i_1(\mu)$, by \autoref{Lemma: club shooting} we have $N_1\models$“$\mathbb S$ has a dense subset that is $t_k(\lambda)^+$-directed closed”.
    By Claim 2.1.1, $\mathbb S$ preserves $\mathrm{GCH}$, so for any $(N,\mathbb S)$-generic filter $I$, $N_1[I]\models t_k(\lambda)\in C^{(2)}$.
    Let $I$ be such a generic filter.
    Since $\kappa$ is not measurable in $N$, $\mathbb R$ is undefined at $\kappa$, and by \autoref{Lemma: club shooting}, $\mathbb R$ has a dense subset $\mathbb D$ that is $\kappa^+$-closed.
    As $N_1\models\mathrm{GCH}$, $\mathbb D$ has only $|i_1(2^\kappa)|=|i_1(\kappa^+)|=\kappa^+$ many dense subsets in $N_1$.
    Because $N_1^\kappa\subseteq N_1$, a standard argument yields a filter $J\in V_1$ that is $\mathbb R$-generic over $N_1$.
    In $N_1[I]$ we have $|\mathbb P\times \mathbb R|=i_1(\kappa)<(\kappa^{++})^{V_1}<\lambda\le t_k(\lambda)$, so $\mathbb P\times \mathbb R$ is a small forcing relative to $t_k(\lambda)$; consequently $t_k(\lambda)$ remains $\Sigma_2$-correct in $N_1[G\times I\times J]$.
    Lift $k$ to $k^*:V_2\to N_1[G\times I\times J]$. We verify that $k^*$ is a tightly $\lambda$-$C^{(2)}$-compact embedding; it suffices to show $t_{k^*}(\lambda)=t_k(\lambda)$.
    Suppose otherwise; then there exists $s\in N_1[G\times I\times J]$ with $|s|^{N_1[G\times I\times J]}<t_k(\lambda)$ and $k^{*}{''}\lambda\subseteq s$.
    Since $N_1[G\times J]\models$“$\mathbb S$ has a dense subset that is $(|s|^+)^{N_1[G\times I\times J]}$-distributive”, we have $s\in N_1[G\times J]$.
    Moreover, $N_1\models|\mathbb P\times \mathbb R|<(\kappa^{++})^{V_1}<t_k(\lambda)$, so $N_1[I]\models$“$\mathbb P\times \mathbb R$ is $(\kappa^{++})^{V_1}.c.c$”. Hence there exists $t\in N_1$ with $|t|^{N_1}<t_k(\lambda)$ and $\mathbb{P\times R}\Vdash^{N_1}t\supseteq \dot s$.
    But then $k^{*}{''}(\lambda)=k''\lambda\subseteq s\subseteq t$, contradicting the definition of $t_k(\lambda)$. Therefore $t_{k^*}(\lambda)=t_k(\lambda)\in (C^{(2)})^{N_1[G\times I\times J]}$.

    Finally, we show that the required $(N,\mathbb S)$-generic filter $I$ exists in $V_2$.
    Write $j_1(\mathbb P)=\mathbb P\times\mathbb C_\kappa\times \mathbb P_{tail}$, where $\mathbb P_{tail}=j_1(\mathbb P)\upharpoonright(\kappa,j_1(\kappa))$.
    Clearly $i_1(\mathbb P_{tail})=\mathbb S$.
    Since $M_1\models$“there is no $M$-strongly compact cardinal in $(\kappa,\mu)$ and $\mathrm{cf}(\mu)<\mu$”, $\mathbb P_{tail}$ is trivial on $(\kappa,\mu^+]$, and by \autoref{Lemma: club shooting}, $\mathbb P_{tail}$ has a dense subset that is $\mu^+$-directed closed.
    In $M_1$, $\mathbb P_{tail}$ has only $j_1(\kappa)=|\kappa^{\mu^{<\kappa}}|=\mu^+$ many dense subsets, so there exists $F\in V_1$ that is $(M_1,\mathbb P_{tail})$-generic.
    Then $i_1''F$ generates an $(N_1,\mathbb S)$-generic filter; set $I=i_1''F$. This completes the proof.
\end{proof}

\begin{Remark}
    If $V$ satisfies that there is no strong cardinal above $\kappa$, and $\mathcal K$ denotes the $V$-strong cardinals below $\kappa$, then the conclusion of \autoref{Theorem: tightly C(2) may not} can be strengthened to “$V_2\models$‘the least tightly $C^{(2)}$-compact cardinal has no strong cardinal below it’”.
\end{Remark}

\begin{Theorem}
    In the model $V_2$ obtained in \autoref{Theorem: tightly C(2) may not}, the Mitchell order of $\kappa$ equals $1$. Hence, although a tightly $C^{(2)}$-compact cardinal must be a limit of measurable cardinals, its Mitchell order need not be greater than $1$.
\end{Theorem}

\begin{proof}[Proof]
    Otherwise, suppose there exists a normal measure $U\in V_2$ on $\kappa$ with $o^{V_2}(U)=1$.
    Let $j_U:V_2\to M_U$ be the ultrapower induced by $U$.
    Since $V_2$ is a gap forcing extension of $V$, the restriction $i=j_U\upharpoonright V:V\to N$ is definable in $V$ and $j_U$ is a lift of $i$, so $N^{i(\mathbb Q*\mathbb P)}=M_U$.
    Thus $M_U$ is also a gap forcing extension of $N$, and $\kappa$ is measurable in $M_U$. Consequently $\kappa$ is measurable in $N$, and because $\mathrm{crit}(i)=\kappa$, $\kappa$ is a measurable limit of strongly compact cardinals in $N$; hence $\kappa$ is strongly compact in $N$.
    As in \autoref{Theorem: tightly C(2) may not}, we can show that $\kappa$ remains measurable in $N^{i(\mathbb Q)}$, so $\kappa\in i(\mathcal K)$. Then in $N^{i(\mathbb Q*\mathbb P)}=M_U$, $\kappa$ is not Mahlo, a contradiction.
    Hence the theorem holds.
\end{proof}

The above theorem naturally raises the following questions:
\begin{Question}
    Can one find a model in which the least supercompact cardinal is the least strongly compact and also the least tightly $C^{(2)}$-compact?
\end{Question}

The following is another question by Usuba:
\begin{Question}[Usuba\cite{Usubaprivate}]
    Is it possible to find a strongly compact cardinal which is not tightly $C^{(1)}$-compact?
\end{Question}

\section{How about $n\ge 3$ ?}

Notice that the proof of \autoref{Theorem: tightly C(2) may not} depends on 
``all the $\Sigma_2$-ccorrect cardinals are indestructible".
However, there's no indestructible $\Sigma_3$-correct cardinal by the following fact:
\begin{Theorem}[Bagaria et al., \cite{BagariaHamkinsLaverindestructible}]
Suppose $V_\kappa\prec_{\Sigma_2} V_\lambda$ for some $\lambda\ge \eta$ and $G$ is $(V,\mathbb Q)$-generic for some non-trivial $\kappa$-strategically closed forcing $\mathbb Q\in V_\eta$.
Then for all $\theta\ge \eta$,
$V_\kappa=V[G]_\kappa\not\prec_{\Sigma_3}V[G]_\theta$.
    
\end{Theorem}

Hence, the adoption of new, distinct methods, unlike those previously considered, is essential to handle the case $n\ge 3$.

\section*{Acknowledgment}
The author thanks Joan Bagaria, Gabriel Goldberg and Toshimichi Usuba for their helpful comments and suggestions on this paper.
The author is grateful to Toshimichi Usuba for carefully reading a previous version of this paper and for sharing his private manuscript with the author.
 The author also thanks Daisuke Ikegami for his helpful conversations during the preparation of this work.

\section*{AI declaration}
 The author used AI for translation purposes. No AI tools were used to generate mathematical content, proofs, or results. The author takes full responsibility for the content of this paper.
\normalem
\printbibliography{}

\end{document}